\documentclass{mcom-l}
\usepackage{hyperref,lineno}
\usepackage{lipsum,bm,enumerate}
\usepackage{amsfonts}
\usepackage{graphicx}
\usepackage{epstopdf}
\usepackage{algorithmic}
\usepackage{amsmath,booktabs,ctable,threeparttable}
\usepackage{amssymb,amsfonts,boxedminipage,enumerate}
\usepackage{algorithm}
\usepackage{algorithmic}
\usepackage{multirow}
\usepackage{verbatim}

\numberwithin{equation}{section}
\newtheorem{theorem}{Theorem}[section]
\newtheorem{corollary}[theorem]{Corollary}
\newtheorem{lemma}[theorem]{Lemma}
\newtheorem{assumption}[theorem]{Assumption}
\theoremstyle{definition}
\newtheorem{definition}[theorem]{Definition}
\newtheorem{remark}[theorem]{Remark}

\makeatletter
 \newcommand{\norm}[1]{\left\Vert#1\right\Vert}
\newcommand{\abs}[1]{\left\vert#1\right\vert}

\newcommand{\Unif}[1]{\mathbb{U}(#1)}
\newcommand{\floor}[1]{\lfloor#1\rfloor}
\newcommand{\ceil}[1]{\lceil #1 \rceil}
\def\var{\mathrm{Var}}

\def\BV{\mathrm{HK}}
\def\mrd{\mathrm{d}}
\newcommand{\nat}{\mathbb{N}}
\newcommand{\mbe}{\mathbb{E}}

\begin{document}

\title[QMC for quantile and expected shortfall]{Convergence analysis of quasi-Monte Carlo sampling for quantile and expected shortfall}


\author{Zhijian He}
\address{School of Mathematics, South China University of Technology,  Guangzhou  510641, China}
\curraddr{}
\email{hezhijian@scut.edu.cn}
\author{Xiaoqun Wang}
\address{Department of Mathematical Sciences, Tsinghua University,  Beijing 100084, China}
\email{wangxiaoqun@mail.tsinghua.edu.cn}
\thanks{This work was supported by the National Science Foundation of
	China under Grant No. 71601189 and the National Key R\&D Program of China under Grant No. 2016QY02D0301.}


\date{}


\begin{abstract}
	Quantiles and expected shortfalls are usually used to measure risks of stochastic systems, which are often estimated by Monte Carlo methods. This paper focuses on the use of quasi-Monte Carlo (QMC) method, whose convergence rate is asymptotically better than Monte Carlo in the numerical integration. We first prove the convergence of QMC-based quantile estimates under very mild conditions, and then establish a deterministic error bound of $O(N^{-1/d})$ for the quantile estimates, where $d$ is the dimension of the QMC point sets used in the simulation  and $N$ is the sample size. Under certain conditions, we show that the mean squared error (MSE) of the randomized QMC estimate for expected shortfall  is $o(N^{-1})$.
	Moreover, under stronger conditions the MSE can be improved to $O(N^{-1-1/(2d-1)+\epsilon})$ for arbitrarily small $\epsilon>0$.
\end{abstract}

\maketitle

\section{Introduction}
Many application areas use quantiles or expected shortfalls to measure
risks of stochastic systems. For instance, in the financial industry, a quantile (known as value-at-risk) plays an important role for quantifying and managing portfolio risk. On the other hand, expected shortfall (known as conditional value-at-risk) may provide incentives for risk managers to take into account tail risks beyond quantile. We refer to \cite{hong:2014} for a review on the two measures. This paper focuses on estimating quantiles and expected shortfalls via simulation-based methods. Monte Carlo (MC) is a natural method to estimate  them. However, the MC approach is often criticized for time-consuming, since value-at-risk estimation is often relevant to rare events simulation. That usually calls for a large number of runs to get accurate estimation. To address this issue, various variance reduction techniques are employed to increase the accuracy of MC.  Importance sampling (IS) is a promising variance reduction technique for value-at-risk estimation (see, e.g., \cite{glas:2000,glyn:1996}).

Beyond the use of MC, Avramidis and Wilson \cite{avra:wils:1998} proposed correlation-induction techniques to improve quantile estimation based on Latin hypercube sampling (LHS). They showed that the correlation-inducted LHS estimator is  asymptotically normal and unbiased with smaller variance than that of the crude MC. Subsequently, Jin et al. \cite{jin:2003} modified the correlation-inducted LHS estimator of \cite{avra:wils:1998} and proposed a new quantile estimator based on an indirect means of realizing full stratification of \cite{owen:1998} that reuses samples. They showed that the error probability for the stratified quantile estimator is zero for sufficiently large, but finite sample size $N$. Moreover, in some special cases, the convergence rate  is $O(N^{-1})$, as opposed to the crude MC rate $O(N^{-1/2})$.  However, the stratified quantile estimator requires sample sizes that grow exponentially with the dimension of the problem. 

Quasi-Monte Carlo (QMC) methods are deterministic versions of the MC methods, and have an asymptotically faster convergence rate than MC as shown in the field of numerical integration. 
It is straightforward to use QMC methods for estimating quantiles and expected shortfalls. Particularly, Papageorgiou and Paskov \cite{papa:pask:1999} observed from empirical studies that QMC methods provide a highly efficient alternative to MC for quantile calculation. Jin and Zhang \cite{jin:zhang:2006} aimed at smoothing QMC estimators via Fourier transformation so that the faster convergence rate of QMC methods can be reclaimed. To the best of our knowledge, the convergence and the rates of convergence for plain QMC in estimating quantile and expected shortfall are still unclear. 

In this paper, we focus on the use of QMC and randomized QMC (RQMC) for  estimating quantile and expected shortfall. We first prove the convergence of  QMC-based quantile estimates, and establish some useful error bounds for assessing the error rate. We then provide an error bound for the expected shortfall estimate, and find that the efficiency of the expected shortfall estimate is strongly tied to the efficiency of (R)QMC quadrature for a specific discontinuous function and a specific  function with kinks. Under mild conditions, we show that the mean squared error (MSE) of the RQMC-based expected shortfall estimate is $o(N^{-1})$, which is asymptotically better than plain MC and LHS.
Moreover, under stronger conditions the MSE can be improved to $O(N^{-1-1/(2d-1)+\epsilon})$ for arbitrarily small $\epsilon>0$, where $d$ is the dimension of the problem.

The rest of this paper is organized as follows. In Section~\ref{sec:defns}, we introduce some background on quantile estimation and some preliminary results on QMC methods. In Section~\ref{eq:main}, we study the convergence and the convergence rate of  QMC-based quantile estimate. In Section~\ref{sec:cvar}, we study the MSE of  QMC-based   expected shortfall estimate. In Section~\ref{sec:num}, we 
perform a numerical study for stochastic network models for which our theoretical results can be applied.  Section~\ref{sec:concl} concludes this paper.

\section{Preliminaries}\label{sec:defns}
Let $X$ be a real-valued random variable of interest with a cumulative distribution function (CDF)  $F(x)$. For instance, $X$ is the loss or profit of a portfolio over a given holding period. We are interested in the left tail of the distribution of $X$. For a fixed $p\in(0,1)$, the quantity
\begin{equation}\label{eq:defvar}
v := F^{-1}(p)=\inf\{x\in \mathbb{R}|F(x)\geq p\},
\end{equation}
is called the $p$'th quantile of $X$ (or the value-at-risk of $X$ in the context of risk management). The expected shortfall of $X$ is defined as
\begin{equation}\label{eq:defcvar}
c := v-\frac 1 p\mbe[(v-X)^+],
\end{equation}
where $x^+:=\max\{x,0\}$. The expected shortfall is also known as the tail
conditional expectation or conditional value-at-risk. Assume that the variable $X$ can be simulated easily. Our goal is to estimate the quantile $v$ and the expected shortfall $c$ by means of simulation.

In the MC setting, the CDF $F(x)$ of $X$ can be estimated by the empirical CDF
\begin{equation}\label{eq:ecdf}
\hat{F}_N(x)=\frac 1N\sum_{i=1}^N\bm{1}\{X_i\leq x\},
\end{equation}
where $X_i$'s are independent and identically distributed random replications of $X$.  The quantile $v$  is then estimated by
\begin{align}
\hat{v}_N &= \hat{F}^{-1}_N(p)=\inf\{x\in \mathbb{R}|\hat F_N(x)\geq p\}.\label{eq:est}
\end{align}
Let $X_{(i)}$ be the $i$th-order statistic of $X_1,\dots,X_N$. It is easy to see that $\hat{v}_N = X_{(\ceil{pN})}$, where $\ceil{x}$ denotes the smallest integer no less than $x$. The corresponding estimate of the expected shortfall $c$ is given by
\begin{equation}
\hat{c}_N = \hat{v}_N-\frac{1}{pN}\sum_{i=1}^N(\hat{v}_N-X_i)^+.\label{eq:estCVaR}
\end{equation}

Serfling \cite{serf:1980} showed that $\hat{v}_N\to v$ with probability 1 (w.p.1) as $N\to\infty$ under very mild assumptions. If $X$ has a continuous density $f_X(\cdot)$ in a neighborhood of $v$ and $f_X(v)>0$, Serfling \cite{serf:1980} further showed that $\hat{v}_N$ is asymptotically normally distributed. For the expected shortfall estimate \eqref{eq:estCVaR}, Trindade et al.
\cite{trin:2007} found that under certain conditions, $\hat{c}_N\to c$ w.p.1 as $N\to\infty$, and $\hat{c}_N$ is asymptotically normally distributed.

MC is often criticized for its slow convergence. QMC has the potential to improve the convergence rate. We now turn to the regime of QMC in estimating quantiles and expected shortfalls. To start with, let's consider the problem of estimating an integral over the unit cube $[0,1)^d$
\begin{equation*} 
I(f) = \int_{[0,1)^d}f(\bm{u})\mrd \bm{u}.
\end{equation*}
QMC quadrature rule takes the average
\begin{equation}\label{eq:deterestimate}
\hat{I}_N(f)= \frac 1 N \sum_{i=1}^N f(\bm{u}_i),
\end{equation}
where $\bm{u}_1,\dots,\bm{u}_N$ are carefully chosen points in $[0,1)^d$.  The Koksma-Hlawka inequality gives a deterministic error bound for the quadrature rule \eqref{eq:deterestimate}
\begin{equation}\label{K-H}
\abs{\hat{I}_N(f)-I(f)}\leq V_{\BV}(f)D^{*}_N(\mathcal{P}),
\end{equation}
where $\mathcal{P}:=\{\bm{u}_1,\dots,\bm{u}_N\}$, $V_{\BV}(f)$ is the variation of $f(\bm{u})$ in the sense of Hardy and Krause, and $D^{*}_N(\mathcal{P})$ is the star-discrepancy of points in $\mathcal{P}$; see \cite{nied:1992} for details. There are many ways to construct point sets such that $D^{*}_N(\mathcal{P})=O(N^{-1}(\log N)^d)$. As a result, the QMC error is $O(N^{-1}(\log N)^d)$ for integrands with bounded variation in the sense of Hardy and Krause (BVHK). In this paper, we restrict our attention to $(t,d)$-sequences or $(t,m,d)$-nets in base $b\geq 2$ (see the definitions below). 

\begin{definition}\label{defn1}
	An elementary interval in base $b$ is a subset of $[0,1)^d$ of the form
	\begin{equation}\label{eq:EI}
	E= \prod_{j=1}^d\bigg[\frac{t_j}{b^{k_j}},\frac{t_j+1}{b^{k_j}}\bigg),
	\end{equation}
	where $k_j\in \nat $, $t_j\in \nat $ with $t_j<b^{k_j}$ for $j=1,\dots,d$.
\end{definition}

The elementary interval \eqref{eq:EI} is a hyperrectangle of volume $b^{-\sum_{j=1}^d k_j}$. For given $k_j$, the unit cube $[0,1)^d$ is partitioned into  $b^{\sum_{j=1}^d k_j}$ elementary intervals of the form \eqref{eq:EI}. 

\begin{definition}\label{defnnets}
	Let $t$ and $m$ be nonnegative integers with $t\leq m$. A point set of $b^m$ points $\bm{u}_1,...,\bm{u}_{b^m} \in [0,1)^d$ is a $(t,m,d)$-net in base $b$ if every elementary interval in base $b$ of volume $b^{t-m}$ contains exactly $b^t$ points of the point set.
\end{definition}
\begin{definition}
	Let $t$ be a nonnegative integer. An infinite sequence $\bm u_i\in [0,1)^d$ is a $(t,d)$-sequence in base $b$ if  the finite point set $\bm{u}_{kb^m+1},...,\bm{u}_{(k+1)b^m}$ is a $(t,m,d)$-net in base $b$ for all $k\geq 0$ and $m\geq t$.
\end{definition}

For QMC, it is important to obtain an estimate of the quadrature error $|\hat{I}_N(f)-I(f)|$. But both the variation and the star discrepancy in the upper bound \eqref{K-H} are very hard to compute, and the upper bound is restricted to functions of finite variation. Instead, one can randomize the  points $\bm{u}_1,\dots,\bm{u}_N$ and treat the random version of the quadrature $\hat{I}_N(f)$ in \eqref{eq:deterestimate} as an RQMC quadrature rule.
Usually, the randomized points are uniformly distributed over $[0,1)^d$, and the low discrepancy property of the points is preserved under the randomization (see \cite{lecu:lemi:2005} and Chapter 13 of the monograph \cite{dick:pill:2010}  for a survey of various RQMC methods). In this paper, we focus on the scrambling technique proposed by \cite{owen:1995} to randomize  $(t,d)$-sequences or $(t,m,d)$-nets.

QMC methods are designed to sample $d$-dimensional vectors that are uniformly distributed on the unit cube $[0,1)^d$. To fit into the setting of quantile estimation, one needs to know the mechanism of sampling the target variable $X$ via standard uniform distributed variables. In what follows, we assume that the target variable $X$ can be generated by 
\begin{equation}\label{eq:map}
X=\phi(\bm u ),\ \bm u\sim \Unif{[0,1)^d},
\end{equation}
where the function $\phi:[0,1)^d\to \mathbb{R}$ is easily computed. For practical problems, it may be easy to obtain the mapping $\phi$ by using the multivariate inverse transformation proposed by Rosenblatt \cite{Rose:1952}. In the QMC setting, we shall rewrite the empirical CDF \eqref{eq:ecdf} as
\begin{equation}\label{eq:qmcecdf}
\hat{F}_N(x) = \frac 1 N\sum_{i=1}^N\bm 1\{\phi(\bm u_i)\le x\},
\end{equation}
where $\bm{u}_1,\dots,\bm{u}_N$ are QMC or RQMC points. The QMC estimate of the quantile $v$  is then obtained by the formula \eqref{eq:est}. The expected shortfall estimate \eqref{eq:estCVaR} is then replaced by

\begin{equation}
\hat{c}_N = \hat{v}_N-\frac{1}{pN}\sum_{i=1}^N(\hat{v}_N-\phi(\bm u_i))^+.\label{eq:es}
\end{equation}
Note that the empirical CDF \eqref{eq:qmcecdf} can be viewed as a QMC quadrature rule $\hat{I}_N(f)$ for the indicator function $f(\bm u)=\bm 1\{\phi(\bm u)\le x \}$. However, the upper bound \eqref{K-H} of the error does not provide useful information for this case because the variation $V_{\BV}(f)$ is usually infinite for discontinuous integrands \cite{owen:2005}. 

\section{Convergence analysis for QMC quantile estimation}\label{eq:main}

In this section, we first show the convergence of QMC estimates for quantile estimation under very mild conditions. Then we give a deterministic error bound for QMC estimates under some relatively stronger conditions.

\begin{theorem}\label{thm:qmcconsistency}
	Let $X$ be a random variable with CDF $F(x)$. The empirical CDF $\hat F_N(x)$ given by \eqref{eq:qmcecdf} is based on QMC points. Assume that 
	\begin{enumerate}[(i)]
		\item $v=F^{-1}(p)$ is the unique solution $x$ of $F(x-)\le p\le F(x)$, and
		\item $\lim_{N\to\infty}\hat F_N(x)=F(x)$ for all $x\in\mathbb{R}$.
	\end{enumerate}
	Then $\hat{v}_N=\hat F_N^{-1}(p)\to v$ as $N\to \infty$.
\end{theorem}

\begin{proof}
	By the definition of $\hat v_N$ in \eqref{eq:est}, we have $$ p\le \hat F_{N}(\hat{v}_{N})\le p+1/N$$ for all $N\ge 1$. By the uniqueness condition (i) and the definition of $v$ in \eqref{eq:defvar}, we find that for any $\eta>0$,
	\begin{equation}
	F(v-\eta)<p<F(v+\eta).\label{eq:unique}
	\end{equation}
	
	Assume that $\hat{v}_N$ does not converge to $v$. Then there exists an $\epsilon>0$ and an infinite sequence of positive integers $n_i$ with $\lim_{i\to\infty}n_i=\infty$ such that
	$|\hat{v}_{n_i}-v|\ge \epsilon$ for all $i$. If $\hat{v}_{n_i}\ge v+\epsilon$, then
	$\hat F_{n_i}(v+\epsilon)\le \hat F_{n_i}(\hat{v}_{n_i})\le p+1/n_{i}.$ By condition (ii),
	$$F(v+\epsilon)=\lim_{i\to\infty} \hat F_{n_i}(v+\epsilon)\le p.$$
	This leads to a contradiction because $F(v+\epsilon)>p$ by using \eqref{eq:unique}. 
	
	On the other hand, if $\hat{v}_{n_i}\le v-\epsilon$, then $\hat F_{n_i}(v-\epsilon)\ge \hat F_{n_i}(\hat{v}_{n_i})\ge p$. By condition (2),
	$$F(v-\epsilon)=\lim_{i\to\infty} \hat F_{n_i}(v-\epsilon)\ge p.$$
	That also leads to a contradiction because $F(v-\epsilon)<p$. As a result, $\hat v_N$ converges to $v$ as $N$ goes to infinity. 
\end{proof}

The uniqueness condition (i) is also the minimal requirement for establishing the strong consistency of the associated MC estimate \cite[p. 75]{serf:1980}. Condition (ii) implies that the empirical CDF $\hat F_N(x)$ converges to the true CDF $F(x)$ for all $x\in \mathbb{R}$ in the QMC setting. Recall that the random variable $X$ can be generated via the mapping \eqref{eq:map}.
This calls for the Jordan measurability of the set
\begin{equation*}\label{eq:omegax}
\Omega_x:=\{\bm u\in[0,1)^d|\phi(\bm u)\le x \}
\end{equation*}
for all $x\in \mathbb{R}$.

\begin{corollary}\label{thm:rqmcconsistency}
	Suppose that the point set $\{\bm u_1,\dots,\bm u_N\}$ used in \eqref{eq:qmcecdf}  is the first $N$ points of a $(t,d)$-net in base $b\ge 2$. If $v=F^{-1}(p)$ is the unique solution $x$ of $F(x-)\le p\le F(x)$ and $\Omega_x$ is Jordan measurable for all $x\in \mathbb{R}$, then $\hat{v}_N\to v$ as $N\to \infty$. 
\end{corollary}

\begin{proof}
	Let $f(\bm u)=\bm{1}\{\phi(\bm u)\le x\}=\bm{1}\{\bm u\in\Omega_x\}$. Since $\Omega_x$ is Jordan measurable, $f(\bm u)$ is Riemann integrable. Note that $\hat{I}_N(f)=\hat F_N(x)$ and  $I(f)=F(x)$. 
	It is known that the QMC quadrature $\hat{I}_N(f)$ converges to $I(f)$ for all Riemann-integrable functions $f$ (see, e.g., \cite{nied:1992}). This implies that $\lim_{N\to\infty}\hat F_N(x)=F(x)$ for all $x\in\mathbb{R}$. Applying Theorem~\ref{thm:qmcconsistency} completes the proof.
\end{proof}

Note that Riemann integrability of $\phi(\bm u)$ may not lead to Riemann integrability of the indicator function $1\{\phi(\bm u)\le x \}$. Chen et al. \cite{chen:2011} gave such an example by using Thomae's function. By Lebesgue's theorem (see \cite{mars:1993}), $\bm 1\{\phi(\bm u)\le x \}$ is Riemann integrable (or equivalently, $\Omega_x$ is Jordan measurable) iff $\lambda_d(\partial \Omega_x)=0$, where $\lambda_d(\cdot)$ is the Lebesgue measure in $\mathbb{R}^d$. 

\begin{definition}\label{parbody}
	For a set $A\subset \mathbb{R}^d$, the outer parallel body  of $A$  at distance $\epsilon$ is defined as
	$$(A)_\epsilon:=\{\bm{x}\in \mathbb{R}^d|\norm{\bm{x}-\bm{y}}_2\leq \epsilon\ \mathrm{for \ some}\ \bm{y}\in A\},$$
	where $\norm{\cdot}_2$ denotes the Euclidean norm. When $A=\varnothing$, we use a convention that $(A)_\epsilon=\varnothing$ for any $\epsilon>0$.
\end{definition}

Let $g=g(\epsilon)$ a positive
nondecreasing function defined for all $\epsilon>0$ and satisfying $\lim_{\epsilon\to 0_+}g(\epsilon)=0$. Then we let $\mathcal{M}_g$ be the family of all Lebesgue-measurable $\Omega\subset[0,1]^d$ for which
\begin{equation*}
\lambda_d((\partial \Omega)_\epsilon)\le g(\epsilon) \text{ for all }\epsilon>0.
\end{equation*}
Every $\Omega\in \mathcal{M}_g$ is actually Jordan measurable. Conversely, every Jordan measurable subset of $[0,1)^d$ belongs to $\mathcal{M}_g$ for a suitable function $g$ (see \cite[pp. 168-169]{nied:1992}). To establish an error bound of quantile estimate, we need a stronger condition that $\lambda_d((\partial \Omega_x)_\epsilon)$ has a common upper bound $g(\epsilon)$ for $x$ in a neighborhood of $v$. For $\delta>0$, denote $B(x,\delta):=\{t\in \mathbb{R}||t-x|\le \delta\}$ as a $\delta$-neighborhood of $x$.

\begin{assumption}\label{assum:density}
	Assume that $X$ has a  density $f_X(x)$ in a neighborhood of $v$ and $f_X(x)$ is positive and continuous at $v$.
\end{assumption}
\begin{assumption}\label{assum:bound}
	Assume that there exist a positive
	nondecreasing function $g(\epsilon)$ satisfying $\lim_{\epsilon\to 0_+}g(\epsilon)=0$ and $\delta, \epsilon_0>0$ such that  $$\sup_{x\in B(v,\delta)}\lambda_d((\partial \Omega_x)_\epsilon)\le g(\epsilon)$$ for any $\epsilon\le\epsilon_0$.
\end{assumption}

\begin{theorem}\label{thm:varerror}
	Suppose that the point set $\{\bm u_1,\dots,\bm u_N\}$ used in \eqref{eq:qmcecdf}  is a $(t,m,d)$-net in base $b\ge 2$, where $N=b^m$. If Assumptions~\ref{assum:density} and \ref{assum:bound} are satisfied, then 
	\begin{equation*}
	\abs{\hat{v}_N- v}\le\frac{2g(\sqrt{d}b^{1+t/d}N^{-1/d})}{f_X(v)}
	\end{equation*}
	for $N$ large enough.
\end{theorem}

\begin{proof}
	By Definition~\ref{defnnets}, there exist $K=b^{m-t}$ disjoint elementary \mbox{intervals} $E_1,\dots,E_K$ with volume $b^{t-m}$ such that all $E_k$ contain exactly $b^t$ points of the $(t,m,d)$-net in base $b$.
	Let $\mathcal{T}:=\{k=1,\dots,K|E_k\cap \partial\Omega_x\neq \varnothing\}$, and denote $\#(A)$ as the number of the points of the $(t,m,d)$-net contained in the set $A$. By the fairness of the elementary intervals, we have 
	\begin{align*}
	\abs{\hat{F}_N(x)-F(x)}&=\abs{\frac{1}{N}\sum_{i=1}^{N}\bm 1\{\bm u_i\in \Omega_x\}-\lambda_d(\Omega_x)}\\
	&=\abs{\sum_{k\in \mathcal{T}}\frac{\#(E_k\cap \Omega_x)}{N}-\sum_{k\in \mathcal{T}}\lambda_d(E_k\cap \Omega_x)}\\
	&\le \max\left\lbrace\sum_{k\in \mathcal{T}}\frac{\#(E_k\cap \Omega_x)}{N},\sum_{k\in \mathcal{T}}\lambda_d(E_k\cap \Omega_x)\right\rbrace\\
	&\le \frac{b^t\abs{\mathcal{T}}}{N}.
	\end{align*}
	Similarly to the proof of Lemma 4.1 in \cite{he:wang:2015}, one can choose $E_k$ with length as small as possible. By doing so, the  length of $E_k$ is  no larger than $$\sqrt{d}b^{-\floor{(m-t)/d}}<\sqrt{d}b^{1+t/d}N^{-1/d}=:r(N).$$ Let's assume that $r(N)<\epsilon_0$ by taking large enough $N$. By Assumption~\ref{assum:bound},
	\begin{equation*}
	\abs{\mathcal{T}}\le \frac{\lambda_d((\partial \Omega_x)_r)}{b^{t-m}}\le b^{-t}Ng(r(N)).
	\end{equation*}
	Therefore, we have 
	\begin{equation}
	\sup_{x\in B(v,\delta)}\abs{\hat{F}_N(x)-F(x)}\le g(r(N)).\label{eq:sbounds}
	\end{equation}

	By Assumption~\ref{assum:density}, there exists $\delta'>0$ such that $f_X(x)\ge f_X(v)/2$  for all $x\in B(v,\delta')$.	Let $\epsilon=2g(r(N))/f_{X}(v)$. Since $r(N)\to 0$ as $N\to\infty$, there exists a $N_0(\delta,\delta',\epsilon_0)$ such that $r(N)\le \epsilon_0$ and $\epsilon<\min(\delta,\delta')$ for any $N\ge N_0$. By \eqref{eq:sbounds}, we  have 
	$$\hat{F}_N(v+\epsilon)\ge F(v+\epsilon)- g(r(N)).$$By the mean value theorem, we obtain
	$$F(v+\epsilon)-F(v)=f_X(\xi)\epsilon$$
	for some $\xi\in [v,v+\epsilon]\subset B(v,\delta')$. Since $f_X(\xi)\ge f_X(v)/2$ and $F(v)=\alpha$, 
	$$\hat{F}_N(v+\epsilon)\ge F(v)+f_X(\xi)\epsilon-g(r(N))\ge \alpha+f_X(v)\epsilon/2-g(r(N))=\alpha.$$
	We therefore have $\hat{v}_{N}\le v+\epsilon$. Conversely, we can prove in a similar manner that $\hat{v}_{N}\ge v-\epsilon$. Consequently, $|\hat{v}_{N}- v|\le \epsilon=2g(r(N))/f_X(v)$ all $N\ge N_0$.
\end{proof}

If the set $\Omega_x$  is convex for all $x\in B(v,\delta)$, then $\Omega_x\in \mathcal{M}_g$ with $g(\epsilon)=5d\epsilon$ for $\epsilon$ small enough. This is because the volume $\lambda_d((\partial \Omega)_\epsilon)$ for any convex set $\Omega\subset [0,1)^d$ is bounded by that of the case
$\Omega=[0,1)^d$, which is no larger than $5d\epsilon$ for $\epsilon$ small enough (see the proof of Lemma~\ref{lem:lips}). By Theorem~\ref{thm:varerror}, the deterministic error bound for the QMC-based quantile estimate becomes
\begin{equation*}
\abs{\hat{v}_N- v}\le\frac{10d^{3/2}b^{1+t/d}}{f_X(v)}N^{-1/d}.
\end{equation*}
The result may be extended to pseudo-convex sets (see \cite{zhu:2014}). However, the convex conditions on $\Omega_x$ may be restrictive for practical problems. We next show that under the Lipschitz continuity condition on $\phi$, $g(\epsilon)=\kappa\epsilon$ for some constant $\kappa>0$. The same rate $O(N^{-1/d})$ also applies for this case.


\begin{lemma}\label{lem:lips}
	Suppose that $\phi(\bm u)$ is Lipschitz continuous over $[0,1)^d$ with modulus $L>0$. If Assumption~\ref{assum:density} is satisfied, there exist $\epsilon_0,\delta>0$ such that 
	$$\sup_{x\in B(v,\delta)}\lambda_d((\partial \Omega_x)_\epsilon)\le (5d+3f_X(v)L)\epsilon,$$ for any $\epsilon\le\epsilon_0$.
\end{lemma}

\begin{proof}		
	Let $S_1$ be the boundary of the unit cube $[0,1)^d$, and let $S_2=\{\bm u\in[0,1)^d|\phi(\bm u)=x\}$. Note that $\partial \Omega_x\subset S_1 \cup S_2$. As a result, $$\lambda_d((\partial \Omega_x)_\epsilon)\le \lambda_d((S_1)_\epsilon)+\lambda_d((S_2)_\epsilon\backslash(S_1)_\epsilon).$$
	Note that $\lambda_d((S_1)_\epsilon)\le 2[(1+2\epsilon)^d-1]= 4d\epsilon+O(\epsilon^2)$. So there exists $\epsilon'>0$ such that $\lambda_d((S_1)_\epsilon)\le 5d\epsilon$ for any $\epsilon\le \epsilon'$. For any $\bm v\in (S_2)_\epsilon\backslash(S_1)_\epsilon$, there exists $\bm w\in S_2$ such that $||\bm v-\bm w||\le \epsilon$. Since $\phi(\cdot)$ is Lipschitz, $|\phi(\bm v)-\phi(\bm w)|=|\phi(\bm v)-x|\le L||\bm v-\bm w||\le L\epsilon$. Therefore,
	$$(S_2)_\epsilon\backslash(S_1)_\epsilon\subset S_3:=\{\bm u\in [0,1)^d||\phi(\bm u)-x|\le L\epsilon\}.$$ 
	
	By Assumption~\ref{assum:density}, there exists $\delta>0$ such that $f_X(x)\le (3/2)f_X(v)$  for all $x\in B(v,2\delta)$. Let $\epsilon_0=\min(\epsilon',\delta/L)$. Assume that $\epsilon\le \epsilon_0$ and $x\in B(v,\delta)$. By the mean value theorem, $\lambda_d(S_3) = F(x+L\epsilon)-F(x-L\epsilon)=2L\epsilon f_X(\xi)$ for some $\xi \in B(x,L\epsilon)\subset  B(v,2\delta)$.  Using $f_X(\xi)\le (3/2)f_X(v)$ gives $\lambda_d((\partial \Omega_x)_\epsilon)\le 5d\epsilon + \lambda_d(S_3)\le (5d+3f_X(v)L)\epsilon$ for all $\epsilon\le \epsilon_0$ and all $x\in B(v,\delta)$.
	
\end{proof}

\begin{theorem}\label{thm:lip}
	Suppose that the point set $\{\bm u_1,\dots,\bm u_N\}$ used in \eqref{eq:qmcecdf}  is a $(t,m,d)$-net in base $b\ge 2$, where $N=b^m$. If Assumptions~\ref{assum:density} is satisfied and $\phi(\bm u)$ is Lipschitz continuous over $[0,1)^d$ with modulus $L>0$, then
	\begin{equation*}
	\abs{\hat{v}_N- v}\le \frac{(10d+6f_X(v)L)\sqrt{d}b^{1+t/d}}{f_X(v)}N^{-1/d}
	\end{equation*}
	for  $N$ large enough.
\end{theorem}
\begin{proof}
	By Lemma~\ref{lem:lips} and Theorem~\ref{thm:varerror}, there exists $N_0>0$ such that  for $N\ge N_0$, 
	$$
	\abs{\hat{v}_N- v}\le\frac{2(5d+3f_X(v)L)(\sqrt{d}b^{1+t/d}N^{-1/d})}{f_X(v)}= \frac{(10d+6f_X(v)L)\sqrt{d}b^{1+t/d}}{f_X(v)N^{1/d}}.
	$$
\end{proof}

Assumption~\ref{assum:density} is typically used in establishing the asymptotic normality of $\hat v_N$ in the MC setting. The Lipschitz continuity condition on $\phi$ can be easily verified for some applications; see Section~\ref{sec:num} for the greater detail. 

\section{Convergence analysis for RQMC expected shortfall estimation}\label{sec:cvar}
In this section, we study the MSE of the expected shortfall estimate $\hat c_N$ when using RQMC.
Define $K(x)=\mbe[(x-X)^+]$, which is estimated by
\begin{equation}
\hat K_N(x) = \frac{1}{N}\sum_{i=1}^N (x-X_i)^+ = \frac{1}{N}\sum_{i=1}^N (x-\phi(\bm u_i))^+,
\end{equation}
where $\bm u_i$ are given in \eqref{eq:qmcecdf}.
Note that $\hat K_N(x)$ can be viewed as a QMC quadrature rule $\hat{I}_N(f)$ and $K(x)=I(f)$ for the kink function  $f(\bm u)=(x-\phi(\bm u))^+$. Also, $c =  v- K(v)/p$ and $\hat c_N = \hat v_N-\hat K_N(\hat v_N)/p$. The following lemma gives a relationship between the quantile estimation error and the expected shortfall estimation error.

\begin{lemma}\label{lem:cvar}
	If Assumption~\ref{assum:density} is satisfied, then
	\begin{equation*}
	\hat{c}_N-c = [K(v)-\hat K_N(v)]/p + B_N,
	\end{equation*}
	where 
	\begin{equation*}
	|B_N| \le \frac 1 p|\hat v_N-v|(2/N+|\hat F_N(v)-F(v)|).
	\end{equation*}
\end{lemma}
\begin{proof}
	Let $B_N=\hat{c}_N-c - [K(v)-\hat K_N(v)]/p$. By Equation (12) in \cite{sum:hong:2010}, we find that
	$$|B_N| \le \frac{1}{p} |\hat v_N-v|(2|\hat F_N(\hat v_N)-F(v)|+|\hat F_N(v)-F(v)|).$$
	Under Assumption~\ref{assum:density}, we have $$|\hat F_N(\hat v_N)-F(v)|=|\hat F_N(\hat v_N)-p|\le 1/N,$$
	which completes the proof.
\end{proof}

\begin{theorem}
	Suppose that the point set $\{\bm u_1,\dots,\bm u_N\}$ used in \eqref{eq:qmcecdf} and \eqref{eq:es}  is a scrambled $(t,m,d)$-net in base $b\ge 2$, where $N=b^m$. Suppose additionally that Assumptions~\ref{assum:density} and \ref{assum:bound} are satisfied and $(v-\phi(\bm u))^+\in L^2([0,1)^d)$. Then for $N$ large enough,
	\begin{equation}\label{eq:cvarmse}
	\mbe[(\hat{c}_N- c)^2]\le \frac 2{p^2} \var[\hat K_N(v)]+a_N \left(\frac{4}{N^2}+\var[\hat F_N(v)]\right).
	\end{equation}
	where
	$$a_N:=8\left[ \frac{g(\sqrt{d}b^{1+t/d}N^{-1/d})}{pf_X(v)}\right]^2\to 0\text{ as }N\to \infty,$$ 
	and $g(\cdot)$ is given in Assumption~\ref{assum:bound}.
	Particularly, $\mbe[(\hat{c}_N- c)^2]=o(1/N)$. If $\phi(\bm u)$ is of BVHK and  $g(\epsilon)=\kappa\epsilon$ for some constant $\kappa>0$, then $$\mbe[(\hat{c}_N- c)^2]=O(N^{-1-1/(2d-1)+\epsilon})$$ for arbitrarily small $\epsilon>0$.
	
\end{theorem}
\begin{proof}
	It is known that a scrambled $(t,m,d)$-net is a $(t,m,d)$-net w.p.1 (see \cite{owen:1995}).
	By combining Theorem~\ref{thm:varerror} and Lemma~\ref{lem:cvar}, we have
	
	\begin{align*}
	\mbe[(\hat{c}_N- c)^2] &\le 2\mbe [(K(v)-\hat K_N(v)]/p^2 +2\mbe[B_N^2]\\
	&\le \frac{2\var[\hat K_N(v)]}{p^2} +\left[ \frac{2g(\sqrt{d}b^{1+t/d}N^{-1/d})}{pf_X(v)}\right]^2\left(\frac{8}{N^2}+2\var[\hat F_N(v)]\right).
	\end{align*}
	
	For any square-integrable integrands, as shown in \cite{owen:1997b}, the scrambled net variance with sample size $N$ is $o(1/N)$. This implies $\var[\hat K_N(v)]=o(1/N)$ and $\var[\hat F_N(v)]=o(1/N)$, leading to $\mbe[(\hat{c}_N- c)^2]=o(1/N)$.
	
	Since $g(\epsilon)=\kappa\epsilon$ for some constant $\kappa>0$, it is easy to see that $\partial \Omega_v$ admits ($d-1$)-dimensional Minkowski content (see \cite{he:wang:2015}). By Theorem~3.5 in \cite{he:wang:2015}, we have $\var[\hat K_N(v)]=O(N^{-1-1/(2d-1)+\epsilon})$ for arbitrarily small $\epsilon>0$. By Theorem~4.4 in \cite{he:wang:2015}, we have $\var[\hat F_N(v)]=O(N^{-1-1/d})$. Consequently, $\mbe[(\hat{c}_N- c)^2]=O(N^{-1-1/(2d-1)+\epsilon})$.
	
\end{proof}

\begin{remark}
	The deterministic error bound for the quantile estimate established in Theorem~\ref{thm:varerror} plays an important role in studying the MSE of the expected shortfall estimate.  The convergence result in Theorem~\ref{thm:qmcconsistency} does not help to bound the MSE. Observed from \eqref{eq:cvarmse}, the accuracy of the expected shortfall estimate depends strongly on the RQMC integration of the kink function $(v-\phi(\bm u))^+$. This implies that if the RQMC quadrature rule yields a faster  rate of convergence for the kink function, one can expect a better performance of the expected shortfall estimate. However, if $\phi(\bm u)$ is not of BVHK, it may be hard to predict the MSE rate for the function $(v-\phi(\bm u))^+$ unless using the worst-case rate $o(1/N)$.

\end{remark}

\section{Numerical Examples}\label{sec:num}

A stochastic activity network (SAN) models the time to compute a project having activities with random durations and precedence constraints. Figure~\ref{fig:san} shows an instance of  SAN with $d=15$ activities, which correspond to the edges in the network.  Dong and Nakayama \cite{dong:2017}  studied this model with LHS. Let $Y_i$ denote the time to complete the activity $i$. Assume that the activity durations $Y_i$ are independent exponential random variables $\mathrm{Exp}(\lambda_i)$, i.e., the density of $Y_i$ is given by $p_i(x)=\lambda_i\exp(-\lambda_i x)\bm 1\{x\ge 0\}$, where $\lambda_i>0$. The network in Figure~\ref{fig:san} has $q=10$ paths form nodes $s$ to $t$, denoted by $B_1,\dots,B_{q}$. Specially, $B_1=\{1,4,11,15\}$, $B_2=\{1,4,12\}$, $B_3=\{2,5,11,15\}$, $B_4=\{2,5,12\}$, $B_5=\{2,6,13\}$, $B_6=\{2,7,14\}$, $B_7=\{3,8,11,15\}$, $B_8=\{3,8,12\},$ $B_9=\{3,9,15\},\ B_{10}=\{3,10,14\}$. The time to complete the project can be modeled by the random variable $$X = \max_{i=1,\dots,q}\sum_{j\in B_i} Y_j.$$ 

We are interested in estimating the quantile of $X$. To simulate the model using QMC, we generate  $Y_j=-(1/\lambda_j)\log u_j$ for $j=1,\dots,d$. As a result, $X$ can be expressed a function of $\bm u=(u_1,\dots,u_d)$, denoted by $\phi(\bm u)$. It should be noted that $\phi(\bm u)$ is not Lipschitz continuous. So Theorem~\ref{thm:lip} cannot be applied directly. To circumvent this, we rewrite the set $\Omega_x$ as $\Omega_x = \cap_{i=1,\dots,q} A_i$, where 
\begin{equation*}
A_i = \{\bm u\in[0,1)^d|-\sum_{j\in B_i} \frac {\log u_j}{\lambda_j}\le x\} =\{\bm u\in[0,1)^d|\prod_{j\in B_i} u_j^{1/\lambda_j}\ge e^{-x}\}.
\end{equation*}
Let $\lambda_{\mathrm{max}} = \max_{j=1,\dots,d} \lambda_j$, and let $\phi_i(\bm u) = \prod_{j\in B_i} u_j^{\lambda_{\mathrm{max}}/\lambda_j}$. Then $A_i=\{\bm u\in[0,1)^d|\phi_i(\bm u)\ge e^{-\lambda_{\mathrm{max}}x}\}$. It is easy to see that $\phi_i$ is Lipschitz continuous over $[0,1]^d$ because $\lambda_{\mathrm{max}}/\lambda_j\ge 1$ for all $j$. So by Lemma~\ref{lem:lips} and using $(\partial \Omega_x)_\epsilon \subset \cup_{i=1,\dots,q} (\partial A_i)_\epsilon$, the conditions in Theorem~\ref{thm:varerror} are satisfied with $g(\epsilon)=\kappa\epsilon$ for some constant $\kappa>0$. The QMC error for the quantile estimation is $O(N^{-1/d})$. 

We now study the MSE of the expected shortfall estimate $\hat{c}_N$ when using RQMC. By Theorem~4.4 in \cite{he:wang:2015}, we have $\var[\hat F_N(v)]=O(N^{-1-1/d})$ since $g(\epsilon)=\kappa\epsilon$. Using \eqref{eq:cvarmse}  gives
$$\mbe[(\hat{c}_N- c)^2]\le \frac{2}{p^2} \var[\hat K_N(v)]+O(N^{-1-2/d}).$$
However, since $\phi(\bm u)$ is not of BVHK, the rate $O(N^{-1-1/(2d-1)+\epsilon})$ established in Theorem~3.5 of \cite{he:wang:2015} cannot be applied for the integrand $(v-\phi(\bm u))^+$. Instead, using the worst-case rate $\var[\hat K_N(v)]=o(1/N)$ arrives at $\mbe[(\hat{c}_N- c)^2]=o(1/N)$. This confirms that RQMC performs asymptotically better than MC and LHS for expected shortfall estimation. The MSE rate $o(1/N)$ may be too conservative when $d$ is small.

Figure~\ref{fig:results} shows the numerical results for the SAN model in Figure~\ref{fig:san} with $\lambda_i=1/2$ for $i\le 8$ and $\lambda_i=1$ for $i>8$. In the numerical experiments,
we use  Sobol' points as inputs for QMC-based estimates and scrambled Sobol' points for RQMC-based estimates. The
MSEs in right panel of Figure~\ref{fig:results} are computed based on 100 independent repetitions. Estimation of the errors requires knowing the true value of the quantity being estimated.
Here we use the MC method with a very large sample size (say, $N=10^{9}$) to obtain  accurate estimates
of $v$ and $c$ and treat them as the true values. We consider the case $p=0.1$ for which the true values are $v=2.5446,\ c=2.1596.$
The empirical evidence shows convergence rates of (R)QMC beyond
the crude MC rate of $N^{-1/2}$. Particularly, RQMC yields lower MSEs than MC for both the quantile and the expected shortfall estimations.

\begin{figure}[t]
	\centering
	\caption{A SAN model taken from  Dong and Nakayama \cite{dong:2017}.\label{fig:san}}
	\includegraphics[width=0.6\hsize]{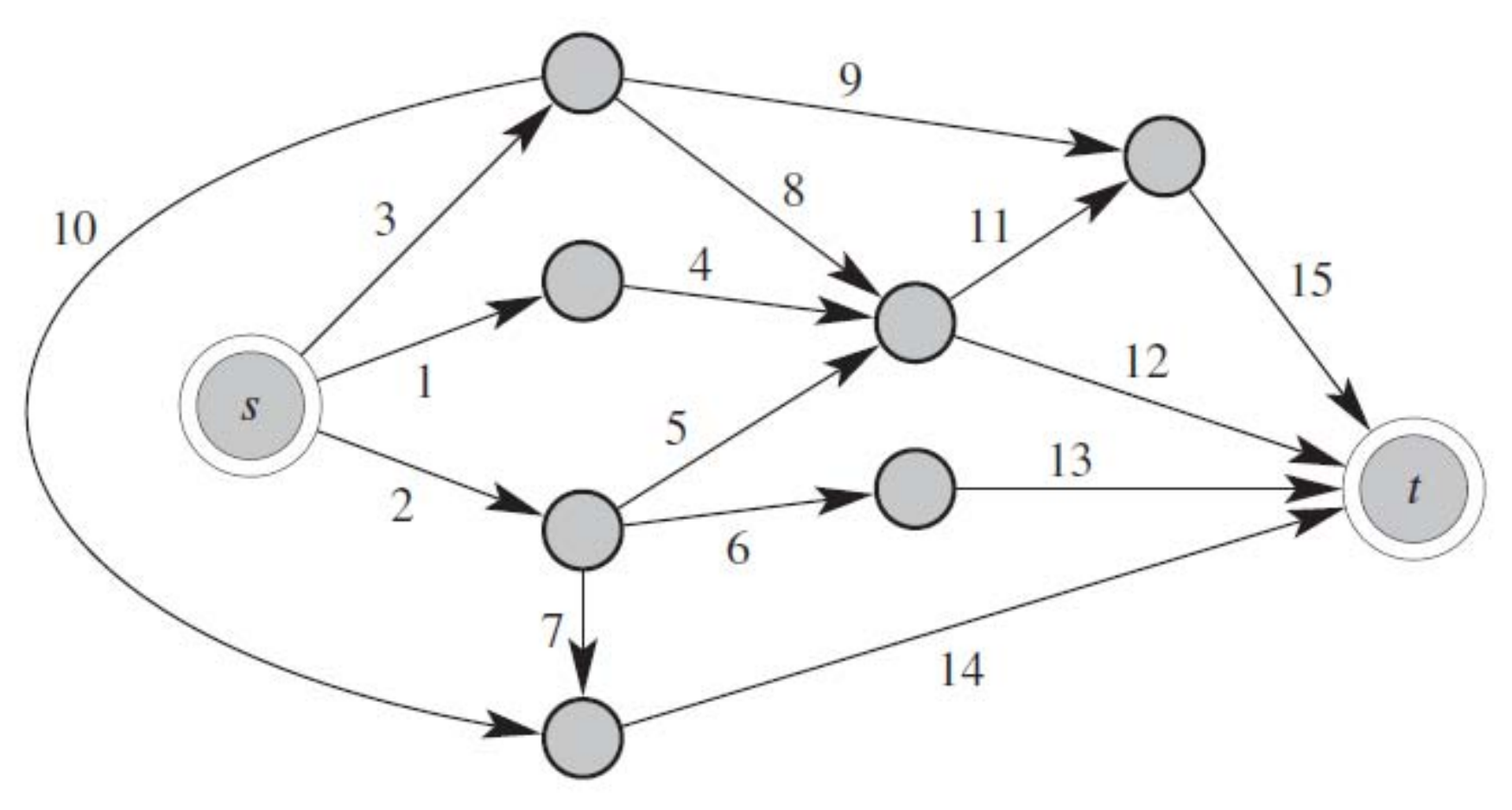}
\end{figure}

\begin{figure}[t]
	\centering
	\caption{The errors of QMC, RQMC and MC based estimates with $p=0.1$. There are two reference
		lines proportional to labeled powers of $N$. All the MSEs are computed based on 100 independent runs for
		$N=2^i,\ i=8,9,\dots,20$. \label{fig:results}}
	\includegraphics[width=\hsize]{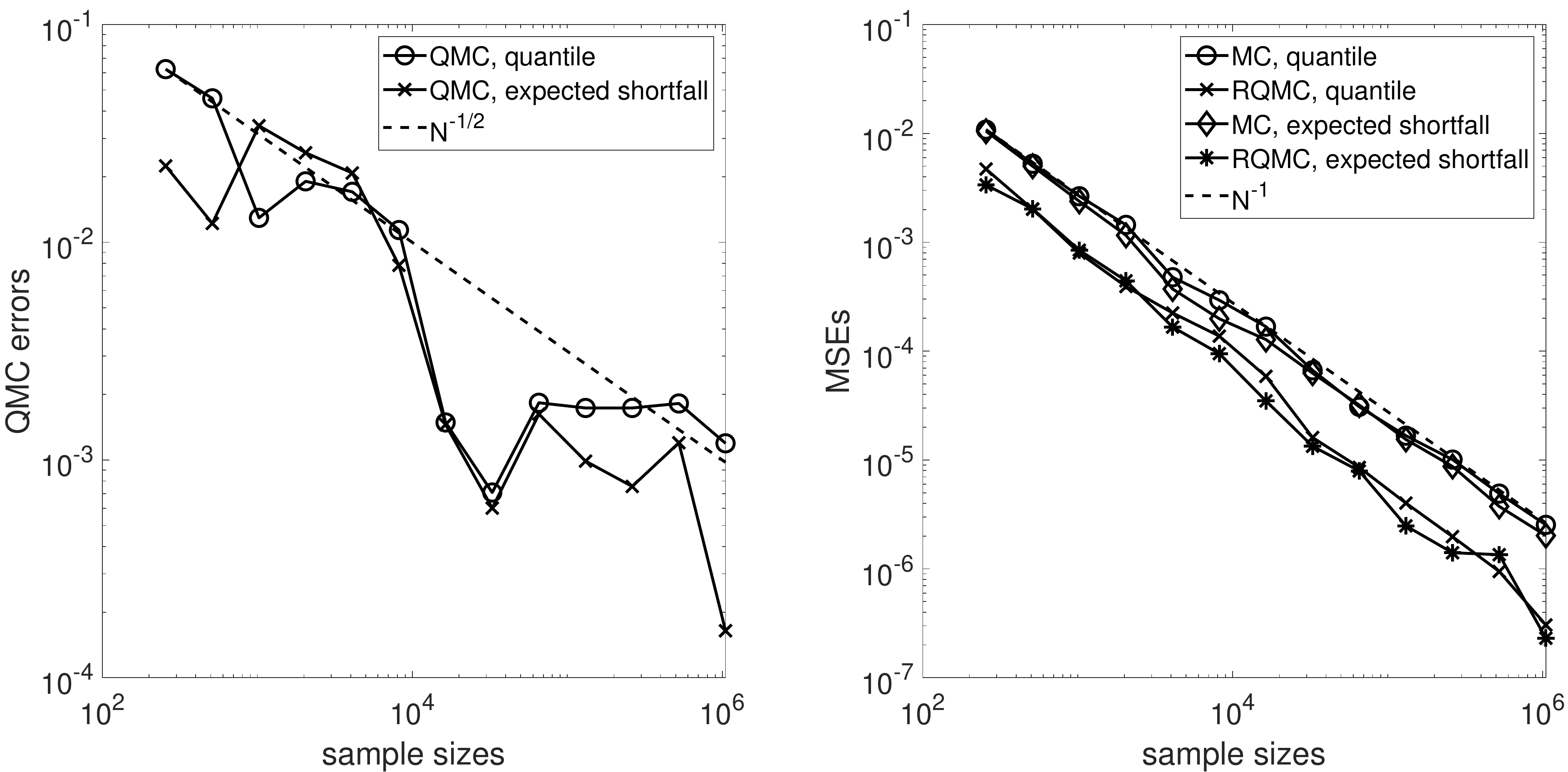}
\end{figure}

\section{Conclusion}\label{sec:concl}
In this paper, we proved the convergence of QMC-based quantile estimates under very mild assumptions. More importantly, we proved that the QMC error  is bounded from above by $N^{-1/d}$. 
The error rate  $O(N^{-1/d})$ is worse than the usual MC rate $O(N^{-1/2})$ for $d>2$. 
But this error rate is critical to establish considerable  MSE rates of RQMC for expected shortfall estimation. It is possible to obtain a faster error rate for RQMC-based quantile estimates as suggested by the numerical study. Owen \cite{owen:1997b} showed that scrambled net quadrature rules can yield an MSE of $o(N^{-1})$ for square-integrable functions. We conjecture that RQMC-based quantile estimation can also lead to an MSE of $o(N^{-1})$ under some technical conditions. We leave this problem open for future research.

\bibliographystyle{amsplain}

\end{document}